\documentclass[12pt]{amsart}
\usepackage{amsmath,amssymb,amsthm,hyperref}
\usepackage{graphicx}
\usepackage{tikz}  
 \newtheorem{theorem}{Theorem}

 \newtheorem{lemma}[theorem]{Lemma}
 \newtheorem{proposition}[theorem]{Proposition}
 
 {\theoremstyle{definition}
 \newtheorem{definition}{Definition}
 \newtheorem{remark}{Remark}

 }
 \newcommand \cal{\mathcal}     
 \newcommand\zu{[0,1]}
\newcommand{\N}{\ensuremath{\mathbb N}} %natural numbers
\newcommand{\R}{\ensuremath{\mathbb R}} %real numbers
\newcommand\CC{\mathcal{CC}^d}
\newcommand{\ep}{\varepsilon}

\newcommand{\xxx}{{\mathbf {x}}}

\newcommand{\zz}{\mathbf {z}}

\usepackage{buczotcv}
\renewcommand{\ggg}{\gamma}

\renewcommand{\lll}{\lambda}
\title{Multifractal properties of typical convex functions}
\author{Zolt\'an Buczolich}
\address{Zolt\'an Buczolich, Department of Analysis,  ELTE E\"otv\"os Lor\'and
University, P\'azm\'any P\'eter S\'et\'any 1/c, 1117 Budapest, Hungary \\
 {\href{http://www.cs.elte.hu/~buczo}{\tt www.cs.elte.hu/\hbox{$\sim$}buczo}}\\
 ORCID: 0000-0001-5481-8797}
\email{buczo@cs.elte.hu} \thanks{Research supported by  the Hungarian National Research, Development and Innovation Office--NKFIH, Grant 124003.}
\author{St\'ephane Seuret }
\address{St\'ephane Seuret, Universit\'e Paris-Est, LAMA (UMR 8050),  UPEMLV, UPEC, CNRS, F-94010, Cr\'eteil, France}
\email{seuret@u-pec.fr }
\thanks{
Research   partly supported by the grant ANR MUTADIS ANR-11-JS01-0009.  
}
\date{\today}
\begin{document}
\maketitle

\medskip

\begin{abstract}
We study the singularity (multifractal)  spectrum of continuous convex functions  defined on $[0,1]^{d}$.  Let  $E_f({h}) $ be the set of points at which $f$ has a pointwise exponent equal to $h$. We first obtain general upper bounds for the Hausdorff dimension of these sets $E_f(h)$, for all convex functions $f$ and all $h\geq 0$.  We prove that for  typical/generic (in the sense of Baire)  continuous convex functions $f:[0,1]^{d}\to { \ensuremath { \mathbb R } }$, one has $\dim E_f(h) =d-2+h$ for all $h\in[1,2],$ and in addition, we obtain that the set $ E_f({h} )$  is empty if $h\in (0,1)\cup (1,+\oo)$. Also, when $f$ is typical, the boundary of $[0,1]^{d}$ belongs to $E_{f}({0})$. 
\end{abstract}

%%%%%%%%%%%%%%%%%%%%%%%%%
%%%%%%%%%%%%%%%%%%%%%%%%%
%%%%%%%%%%%%%%%%%%%%%%%%%
%%%%%%%%%%%%%%%%%%%%%%%%%
%%%%%%%%%%%%%%%%%%%%%%%%%
%%%%%%%%%%%%%%%%%%%%%%%%%
%%%%%%%%%%%%%%%%%%%%%%%%%
%%%%%%%%%%%%%%%%%%%%%%%%%
%%%%%%%%%%%%%%%%%%%%%%%%%
%%%%%%%%%%%%%%%%%%%%%%%%%
%%%%%%%%%%%%%%%%%%%%%%%%%
%%%%%%%%%%%%%%%%%%%%%%%%%
%%%%%%%%%%%%%%%%%%%%%%%%%
\section{Introduction and  main results}\label{introduction}

In this paper we investigate the multifractal properties of the continuous convex functions defined on $[0,1]^{d}$. This paper is a quite natural continuation of our papers \cite{BuS2,BSJMAA} where generic multifractal properties of measures and of  functions monotone increasing   in several variables (in short: MISV) were studied. It is interesting that all these natural objects supported on $\zu^d$ have very different typical multifractal behaviors.

Let us first recall that the  pointwise H\"older exponent and the singularity spectrum  for a locally bounded function are defined as follows.
 
%%%%%%%%%%%%%%%%%%%%%%%%%
\begin{definition}
Let $f \in L^\infty( { [0,1]^ { d } })$. For $h\geq 0$ and $\xxx\in  { [0,1]^ { d } }$,
the function $f$  belongs to $C^h(\xxx)$ if there are a polynomial $P$ of degree strictly  less than $[h]$ and a constant $C>0$ such that,  for  all $\xxx'$ close to  $\xxx$,
\begin{equation}
 \label{defpoint}
|f(\xxx') - P(\xxx'-\xxx)| \leq C |\xxx' -\xxx|^h.
\end{equation}
The pointwise  { H\"older } exponent of $f$ at $\xxx$ is 
$$h_f(\xxx) =  \sup\{h\geq 0: \ f\in C^h(\xxx) \}.$$
 \end{definition}
%%%%%%%%%%%%%%%%%%%%%%%%%

In the following, $\dim=\dim_{H}$ denotes the Hausdorff dimension.

%%%%%%%%%%%%%%%%%%%%%%%%%
\begin{definition}The
 singularity spectrum of  $f$  is the mapping 
 $$d_f(h)=\dim  E_f({h}) , \ \ \mbox{ where }  E_f({h}) =\{\xxx: h_f(\xxx)=h\}.$$
By convention $\dim \emptyset = -\infty$.
We will also use the sets
\begin{equation}
\label{defep}
E_f^{\leq}(h) =\{ \xxx: h_{f}(\xxx)\leq h  \} \supset  E_f({h}) .
\end{equation}
\end{definition}
%%%%%%%%%%%%%%%%%%%%%%%%%

We denote by $ {{\CC}}$ the set of continuous convex functions $f: {[0,1]^d}\to {\ensuremath {\mathbb R}}$.
Equipped with the supremum norm $\|\cdot\|$,  $ {\CC}$ is a separable complete metric space.  An open ball in $ {\CC}$ of center $f\in \CC$ of radius $r\geq 0$ is written as $B_{\|\cdot\|}(f,r)$, and a closed ball is $\overline{ B}_{\|\cdot\|}(f,r)$

In this paper, we first prove an upper bound for the multifractal spectrum of all functions in $\CC$. 
%%%%%%%%%%%%%%%%%%%%%%%%%
\begin{theorem}
\label{mainth1}
For any function $f\in \CC$, one has
$$d_f(h) \leq \begin{cases} \ \  \, d-1 & \mbox{ if } h \in [0,1)\\
d+h-2 & \mbox{ if } h \in [1,2]\\
\ \ \ \ \ d & \mbox{ if } h > 2.
 \end{cases}$$
\end{theorem}
%%%%%%%%%%%%%%%%%%%%%%%%%

Then we compute the multifractal spectrum of typical functions in $\CC$. 
Recall that a  property is  typical, or generic in a complete metric space $E$, when it holds on 
a residual set, i.e. a set  with a complement of first Baire category.

%%%%%%%%%%%%%%%%%%%%%%%%%
\begin{theorem}
\label{mainth2}
For typical functions $f\in \CC$, one has
$$d_f(h) = \begin{cases} \ \  \, d-1 & \mbox{ if } h =0\\
d+h-2 & \mbox{ if } h \in [1,2]\\
 \ \ -\infty & \mbox{ otherwise}.
 \end{cases}$$
 More precisely, one has $E_{f}({0}) = {\partial} ({[0,1]^d})$.
\end{theorem}
%%%%%%%%%%%%%%%%%%%%%%%%%

%
%In this paper in Theorem \ref{*C1dim} we obtain an upper estimate of the multifractal spectrum of any $C^{1}$ function in  ${\CC}$,
%in fact,  we show that $\dim( {E_ {f} ^ {\leq }}(h))\leq d+h-2$
%if $1\leq h\leq 2$. We obtain from this in Theorem \ref{*corupperest}
%that for the typical $f\in {\CC}$  for any $1\leq h\leq 2$ we have $\dim(E_{f}^{\leq }(h))\leq d+h-2.$ 
%
%For the lower estimate of the spectrum in Theorem \ref{*genlowerd}
%we show that for the typical  $f\in {{\cal G}}$ for any $1\leq h \leq2$ we have $\dim E_{f}(h)\geq h+d-2$. For $h>2$, $E_{f}(h)= {\emptyset}$. Moreover, $E_{f}(h)\cap  {(0,1)^d}= {\emptyset}$
%for $0<h<1$ and  $ {\partial} ({[0,1]^d})=E_{f}(0)$.
% 
%  

It is interesting to compare Theorem \ref{mainth2} with the regularity of other  typical objects naturally defined on the cube $\zu^d$.

When $d=1$, generic properties of continuous functions on $\zu$ have been studied for a long time (see for instance \cite{Gruber,Howe}, and the other references we mention in the present paper). It was proved that typical continuous functions belong to $C^0(x)$, for every $x\in \zu$. In \cite{BUC}, typical monotone continuous functions on the interval $\zu$ were proved to have a much more interesting multifractal behavior:  such functions $f:\zu\to\R$ satisfy 
\begin{equation}
\label{inegmonotone}
d_f(h) = \dim E^{\leq}_f(h) =   \begin{cases}   \ \ h & \mbox{ if } h \in [0,1]\\
 -\infty & \mbox{ otherwise}.
 \end{cases}
 \end{equation}
 The same holds true for typical monotone  functions (not necessarily continuous).
  We remark that it also follows, for example from results in  \cite{BUC}, that for arbitrary
 monotone functions  there is an upper estimate
 \begin{equation}
\label{inegmonotoneb}
\dim E^{\leq}_f(h) \leq  h \mbox{ for } h \in [0,1].
 \end{equation}

 It is interesting to extend these results to higher dimensions.

The first natural way is to consider Borel measures on the cube. 
 The local regularity of a positive measure $\mu$  at a given $\xxx\in { [0,1] }^d$ is given by a {local dimension} (or a {local  { H\"older } exponent}) $h_\mu(\xxx)$, defined as
$$%\begin{equation}
 \label{defexpmu}
h_\mu(\xxx)=\liminf_{r\to 0^+}  \frac{\log \mu (B(\xxx,r) )} { \log r},
$$%\end{equation}
where $B(\xxx,r)$ denotes the ball with center $\xxx$ and radius $r$. 
The singularity spectrum of $\mu$ is the map 
$$d_\mu: h\geq 0 \mapsto \dim \,  E_\mu(h),$$
%$\mbox{where }$
where
$%\begin{equation}
% \label{defemuh}
  E_\mu(h):=\{\xxx\in { [0,1]^d }: h_\mu(\xxx)=h\}.
$%  \end{equation}

In  \cite{BuS2} (see also \cite{Bay} for a nice generalization to all compact sets in $\R^d$), it is proved that typical measures $\mu$ supported on $\zu^d$  satisfy a multifractal formalism, and 
$$d_\mu(h) = \begin{cases}   \ \ h & \mbox{ if } h \in [0,d]\\
 -\infty & \mbox{ otherwise}.
 \end{cases}$$

%%%%%%%%%%%%%%%%%%%%%%%%%%%%%%%%%%
\begin{figure}
  \begin{center}
  \includegraphics[width=8.0cm,height=6.4cm]{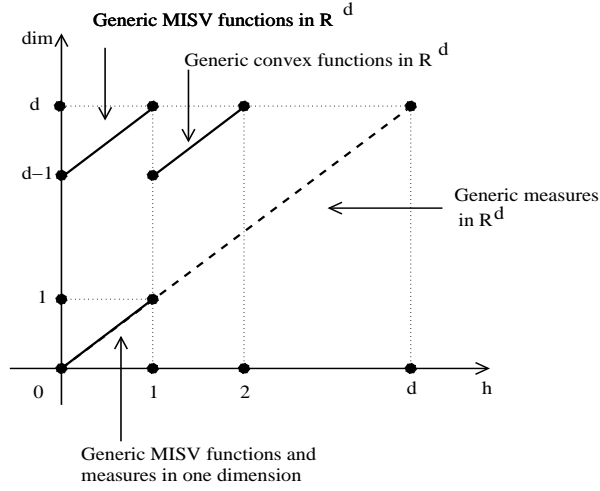}
  
    \caption{Typical spectra for measures, for MISV and for convex functions}
     \label{fig0}
  \end{center}
\end{figure}
%%%%%%%%%%%%%%%%%%%%%%%%%%%%%%%%%%

 Another interesting class is constituted of the  continuous monotone increasing   in several variables (in short: MISV)  functions. These functions extend to higher dimensions in a different direction the one-dimensional monotone  functions. 
 A function $f: { [0,1]^ { d } } \to  { \ensuremath { \mathbb R } }$ is  MISV when  for all $i\in \{1,...,d\}$, the functions 
\begin{equation}
 \label{defgi}
f^{(i)}(t)=
f(x_{1},...,x_{i-1},t,x_{i+1},...,x_{d})
\end{equation}
are continuous monotone increasing. We use the notation
$$ { { \mathcal M } ^d } =\{ f\in C( { [0,1]^ { d } }) : f  \text{  MISV} \}.$$ The  space $ { { \mathcal M } ^d }$ is a separable complete metric space when equipped with the supremum $L^{ { \infty }}$ norm for functions. Typical MISV functions satisfy (see \cite{BUC,BSJMAA})
\begin{equation}
\label{*th1*}
d_f(h) = \begin{cases}  d-1+h & \mbox{ if } h \in [0,1]\\
 \ \ -\infty & \mbox{ otherwise}.
 \end{cases}
 \end{equation}

%%%%%%%%%%%%%%%%%
%\begin{theorem} \label{*th1*}
%For all $f\in { { \mathcal M } ^d }$ and $h\geq 0$,
%we have 
%\begin{equation} \label{*2*1}
% \dim_{H}  E_f^{\leq}(h)   \leq \min(d-1+h,d).
%\end{equation}
%In particular, $ d_{f}(h)=\dim_{H}( E_f({h})) \leq \min(d-1+h,d).$
%\end{theorem}
%%%%%%%%%%%%%%%%%
%
%
%%%%%%%%%%%%%%%%%
%\begin{theorem} \label{*thgensp}
%There exists a dense $G_{ { \delta }}$ set $ { { \mathcal R } } {  \subset } { { \mathcal M } ^d }$ such that for all $f\in { { \mathcal R } }$ we have $d_{f}(h)=d-1+h$ for all $h\in[0,1].$
%For these functions, for every $h>1$ the set $ E_f({h}) $
%is empty.
%\end{theorem}
%%%%%%%%%%%%%%%%%

In Figure \ref{fig0} we  compare our new results about generic
continuous convex functions with the earlier results  we mentioned.

\begin{remark}
\label{rk1}
One cannot directly infer Theorems \ref{mainth1} and \ref{mainth2} by integrating MISV functions or measures on $\zu^d$. For instance, letting $f(x_{1},x_{2})=10(x_1^2+x_2^2)+x_{1}^{2}x_{2}^{2}$,  its second differential $d^{2}f(x_{1},x_{2},h_{1},h_{2})=(20+2x_{2}^{2})h_{1}^{2}+(20+2x_{1}^{2})h_{2}^{2}+4x_{1}x_{2}h_{1}h_{2}$ is positive definite for any $(x_{1},x_{2})\in [-1,1]^{2}$. Hence this function $f$ is strictly convex on $[-1,1]^{2}$,
but $ {\partial}_{1}f=20x_{1}+2x_{1}x_{2}^{2}$ is monotone only in $x_{1}$ and $ {\partial}_{2}f=20x_{2}+2x_{2}x_{1}^{2}$ is monotone only in $x_{2}$.
\end{remark}
%%%%%%%%%%%%%%%%%%%%%%%%%%%%%%%%%%
%%%%%%%%%%%%%%%%%%%%%%%%%%%%%%%%%%
%%%%%%%%%%%%%%%%%%%%%%%%%%%%%%%%%%
%%%%%%%%%%%%%%%%%%%%%%%%%%%%%%%%%%
%%%%%%%%%%%%%%%%%%%%%%%%%%%%%%%%%%
%%%%%%%%%%%%%%%%%%%%%%%%%%%%%%%%%%
%%%%%%%%%%%%%%%%%%%%%%%%%%%%%%%%%%
%%%%%%%%%%%%%%%%%%%%%%%%%%%%%%%%%%
%%%%%%%%%%%%%%%%%%%%%%%%%%%%%%%%%%
%%%%%%%%%%%%%%%%%%%%%%%%%%%%%%%%%%
%%%%%%%%%%%%%%%%%%%%%%%%%%%%%%%%%%
%%%%%%%%%%%%%%%%%%%%%%%%%%%%%%%%%%
%%%%%%%%%%%%%%%%%%%%%%%%%%%%%%%%%%
%%%%%%%%%%%%%%%%%%%%%%%%%%%%%%%%%%
\section{Preliminary Results}

%%%%%%%%%%%%%%%%%%%%%%%%%%%%%%%%%%
%%%%%%%%%%%%%%%%%%%%%%%%%%%%%%%%%%
%%%%%%%%%%%%%%%%%%%%%%%%%%%%%%%%%%
%%%%%%%%%%%%%%%%%%%%%%%%%%%%%%%%%%
%%%%%%%%%%%%%%%%%%%%%%%%%%%%%%%%%%
\subsection{Basic notation}

If it is not stated otherwise we work in $ {\ensuremath {\mathbb R}}^{d}$. Points in the space are denoted by $ \xxx=(x_{1},...,x_{d})$. The $j$'th unit vector is denoted by 
$$ {{\mathbf e}}_{j}=(0,...,0,\underset{\underset{j}{\uparrow}}{1},0,...,0).$$

 Open balls in $\R^d$ are denoted by $B(\xxx,r)$ (to be distinguished from the balls $B_{\|\cdot\|}(f,\ep) $ in $\CC$).

\subsection{Local regularity results of continuous convex functions}

 First,   recall that   $ {{C^ {\infty}}}$ functions are dense in $ {\CC}$.
 
 \begin{remark}
 \label{rk2}
 To see this, take $ {\psi}\geq 0$   a $ {{C^ {\infty}}}$ function, which is $0$ outside ${\bf B}( {{\mathbf 0}},1)$, such that  $  \displaystyle  \int_{{\ensuremath {\mathbb R}}^{d}} {\psi}=1$. Put $ {\psi}_{\lll}=\lll^{d} {\psi}(\frac{1}{\lll} \xxx)$.
If $f$ is convex and continuous on $ {[-1,1]^d}$, then the convolution
${{{\overline {f}}}_{\lll}}( \xxx)=\int_{{[0,1]^d}}f( \xxx) {\psi}_{\lll}( {{\mathbf {y}}}- \xxx)d {{\mathbf {y}}}$ is convex on $[-1+\lll,1-\lll]^{d}$ and $f_{\lll}( \xxx)={{{\overline {f}}}_{\lll}} (\frac{1}{1-\lll} \xxx)$ is convex on $ {[-1,1]^d}$.
Using the uniform continuity of $f$ on $ {[-1,1]^d}$, one   easily sees that
$||f-f_{\lll}||\to0$ as $\lll\to 0.$ One concludes by using an obvious linear transformation. \end{remark}

 A first lemma allows one to control   the left and right  partial derivatives of all functions in a neighborhood of a   convex differentiable function in $\CC$. The notations  ${\partial}_{j,+}f$  and ${\partial}_{j,-}f$  are used for the right and left partial $j$-th derivatives of a convex function $f$.
  
%%%%%%%%%%%%%%%%%%%%%%%%%%%%%%%%%%
\begin{lemma}\label{*djdiffd}
Suppose $f\in {\CC}\cap C^{1} ([0,1]^d)$ and $\ep>0$. There exists   $\varrho_{f,{\varepsilon}}>0$, such that for all 
$j\in  {\{1,...,d \}}$,  if $g\in B_{\|\cdot\|}(f,\varrho_{f,{\varepsilon}})$, then for every $x_{j}\in[ {\varepsilon},1- {\varepsilon}]$, $x_{i}\in[0,1]$, $i\in {\{1,...,d \} \setminus \{j \}}$ we have
\begin{equation}\label{*R3*b}
| {\partial}_{j,\pm} g( {{x_ {1} ,...,x_ {d}}})- {\partial}_{j}f( {{x_ {1} ,...,x_ {d}}})|< {\varepsilon}.
\end{equation} 
\end{lemma}
%%%%%%%%%%%%%%%%%%%%%%%%%%%%%%%%%%

%%%%%%%%%%%%%%%%%%%%%%%%%%%%%%%%%%
\begin{proof} It is enough to fix one $j \in \{1,...,d\}$. 
Recall that  $f\in {\CC}\cap C^{1} ({[0,1]^d})$ implies that $ {\partial}_{j}f$ is non-decreasing in $x_{j}$ and is uniformly continuous on $ {[0,1]^d}$.

Using this uniform continuity, there exists  a partition $0=x_{j,0}<x_{j,1}<...<x_{j,K}=1$
 such that $x_{j,1}<\varepsilon$  and for every $x_{i}\in [0,1]$ with $i\in {\{1,...,d \} \setminus \{j \}}$, one has for every $l=1,...,K$, 
 \begin{equation}\label{*R3*a}
  {\partial}_{j}f(x_{1},...,x_{j-1}, x_{j,l},x_{j+1},...,x_{d}) -  {\partial}_{j}f(x_{1},...,x_{j-1}, x_{j,l-1},x_{j+1},...,x_{d})
 <\frac{{\varepsilon}}{4}.
 \end{equation}

Set 
\begin{equation}\label{*R4*b}
\varrho_{f, \ep,j}=\frac{{\varepsilon}}{4}\min_{l=2,..., K}(x_{j,l}-x_{j,l-1}).
\end{equation}

Consider any function $g\in B_{\|\cdot\|}(f,\varrho_{f, \ep,j})$, and fix $ \xxx=( {{x_ {1} ,...,x_ {d}}}) \in \zu^d$, so that $x_{j}\in[ {\varepsilon},1- {\varepsilon}]$.

There exists an integer $l\in \{1,..., K-1\}$ such that  $x_{j}\in [x_{j,l},x_{j,l+1}].$

Set  $ \xxx(l')=(x_{1},...,x_{j-1}, x_{j,l'},x_{j+1},...,x_{d}).$

Since the two functions $ t\mapsto g(x_{1},...,x_{j-1},  t ,x_{j+1},...,x_{d})$ and
$  t\mapsto f(x_{1},...,x_{j-1},  t ,x_{j+1},...,x_{d})$ are   real convex,  one has
\begin{eqnarray}
\nonumber 
 {\partial}_{j,-}  g( \xxx) & \geq   & {\partial}_{j,-} g( \xxx(l))\geq 
\frac{g( \xxx(l))-g( \xxx(l-1))}{x_{j,l}-x_{j,l-1}} 
\\
\nonumber
&\geq & \frac{f( \xxx(l))-f( \xxx(l-1))-2\varrho_{f,{\varepsilon},j}}{x_{j,l}-x_{j,l-1}}\\
 \nonumber
&\geq & \frac{f( \xxx(l))-f( \xxx(l-1))}{x_{j,l}-x_{j,l-1}}-\frac{{\varepsilon}}{2}\\
\label{*R4*a}& \geq&  {\partial}_{j} f( \xxx(l-1))-\frac{{\varepsilon}}{2} .
\end{eqnarray}

Thanks to the  convexity of $  t\mapsto f(x_{1},...,x_{j-1},  t ,x_{j+1},...,x_{d})$,    equation  \eqref{*R3*a}
gives
$$
 {\partial}_{j}f( \xxx)\leq  {\partial}_{j}f( \xxx(l+1))< {\partial}_{j}f ( \xxx(l-1))+2\frac{{\varepsilon}}{4}.
$$

Hence, one can continue \eqref{*R4*a} to obtain
$$ {\partial}_{j,-} g( \xxx)> {\partial}_{j}f( \xxx)-\frac{{\varepsilon}}{2}-\frac{{\varepsilon}}{2}= {\partial}_{j}f ( \xxx)- {\varepsilon}.$$

Moreover, a  similar argument can show that
$$ {\partial}_{j,+}g( \xxx)< {\partial}_{j}f ( \xxx)+ {\varepsilon}.$$ 
This ends the proof, since ${\partial}_{j,-}g( \xxx)\leq {\partial}_{j,+}g( \xxx)$.

The conclusion follows by taking $\varrho_{f,{\varepsilon}} = \min(\varrho_{f,{\varepsilon},j}: j=1,...,d)$.
\end{proof}
%%%%%%%%%%%%%%%%%%%%%%%%%%%%%%%%%%

The one-dimensional version of Lemma \ref{*djdiffd} is stated as follows.

%%%%%%%%%%%%%%%%%%%%%%%%%%%%%%%%%%
\begin{lemma}\label{*djdiffo}
Suppose $f\in  {\cac\cac^{1}}\cap C^{1}([0,1])$. For $ {\varepsilon}>0$   there exist
$ {\varepsilon}>0$ and $\varrho_{{\varepsilon},f}>0$  such that for any $g\in B_{\|\cdot\|}(f,\varrho_{{\varepsilon},f})$
and $x\in[ {\varepsilon},1- {\varepsilon}]$ we have
\begin{equation}\label{*R6*a}
|g'_{\pm} (x)-f'(x)|< {\varepsilon}.
\end{equation}
\end{lemma}
%%%%%%%%%%%%%%%%%%%%%%%%%%%%%%%%%%

Next, one compares the pointwise exponents of a differentiable convex function $f$ and its derivative $f'$. It is a general property that $h_{f'}(x) \leq h_f(x)-1$, for every differentiable $f$. A surprising property is that equality necessarily holds when   $f$ is convex and $h_f(x)\in [1,2)$,

%%%%%%%%%%%%%%%%%%%%%%%%%%%%%%%%%%
\begin{lemma}\label{*hregd}
If $f$ is convex and differentiable on $(a,b)$  and $h_f(x)\in [1,2)$ for some $x\in (a,b)$, then   $h_{f}(x)= h_{f'}(x)+1$.\end{lemma}
%%%%%%%%%%%%%%%%%%%%%%%%%%%%%%%%%%

%%%%%%%%%%%%%%%%%%%%%%%%%%%%%%%%%%
\begin{proof}
It is enough to prove that $h_{f}(x)\leq h_{f'}(x)+1$. Since $h_f(x)\in [1,2)$, necessarily $h=h_{f'}(x)<1$. Hence, there exists  a sequence $(x_{n})_{n\geq 1}$ converging to $ x$  such that 
$|f'(x_{n})-f'(x)|>|x_{n}-x|^{h+\frac{1}{n}}.$
Without limiting generality, suppose that $x_{n}>x$.
By the monotonicity of $f'$, one has $$f'(x_{n})>f'(x)+(x_{n}-x)^{h+\frac{1}{n}}.$$

By convexity of $f$,
$$
f(x_{n})\geq f(x)+f'(x)(x_{n}-x).%=f(x)+f'(x)(x_{n}-x).
$$
Setting $x_{n}'=x+2(x_{n}-x)$, and using  again the convexity of $f$ at $x_{n}$, one gets
\begin{eqnarray*}
f(x_{n}') & \geq &  f(x_{n})+f'(x_{n})(x_{n}'-x_{n}) \\
&\geq &   f(x)+f'(x)(x_{n}-x)+(f'(x)+(x_{n}-x)^{h+\frac{1}{n}})(x_{n}'-x_{n})\\
&\geq &  f(x)+f'(x)(x_{n}'-x)+(x_{n}-x)^{h+\frac{1}{n}}\frac{1}{2}(x_{n}'-x)\\
&= &  f(x)+f'(x)(x_{n}'-x)+\frac{1}{2^{h+1+\frac{1}{n}}}(x_{n}'-x)^{h+1+\frac{1}{n}}.
\end{eqnarray*}
This implies $h_{f}(x)\leq h+1$, hence the result.
\end{proof}
%%%%%%%%%%%%%%%%%%%%%%%%%%%%%%%%%%

One investigates what happens for non-differentiable convex functions.

%%%%%%%%%%%%%%%%%%%%%%%%%%%%%%%%%%
\begin{lemma}\label{*hregd2}
If $f$ is convex  on $(a,b) \subset \R$  and $h_f(x)\in [1,2)$ for some $x\in (a,b)$, then   $\min(h_{f'_+}(x), h_{f'_-}(x)) \leq h_f(x)-1$.\end{lemma}
%%%%%%%%%%%%%%%%%%%%%%%%%%%%%%%%%%

%%%%%%%%%%%%%%%%%%%%%%%%%%%%%%%%%%
\begin{proof}
When $h_{f}(x)=1$ the lemma is obvious. Set $h=h_f(x)>1$, and let $\ep>0$ so that $h-\ep>0$. By definition, there exists $M\in \R$ such that one has
$$|f(x+y)-f(x) - M y|\leq  |y|^{h-\ep}$$
for every small $y$, and there exists a sequence $(y_n)_{n\geq 1}$  converging to zero such that 
$$|f(x+y_n)-f(x) -M y_n|\geq  |y_n|^{h+\ep}.$$
Hence,
$$ |y_n|^{h+\ep-1} \leq   \left| \frac{f(x+y_n)-f(x)}{y_n}  -M \right|\leq  |y_n|^{h-\ep-1}.$$
Since the left and right derivatives $f'_+(x)$ and $f'_-(x)$ both exist, they both equal $M = f'(x)$. 

Assume, without loss of generality, that there are infinitely many positive $y_n$'s. For every $y_n$, 
$$ f'_+(x+y_n) - f'_+(x) \geq  \frac{f(x+y_n)-f(x)}{y_n}  - f'_+(x)   \geq  |y_n|^{h+\ep-1},$$
thus  $h_{f'_+}(x) \leq h+\ep-1$, which gives the result.
\end{proof}
%%%%%%%%%%%%%%%%%%%%%%%%%%%%%%%%%%

We also prove the following proposition, which somehow asserts that a convex function cannot have exceptional isolated directional  pointwise regularity.

%%%%%%%%%%%%%%%%%%%%%%%%%%%%%%%%%%%
%\begin{lemma}\label{points_sphere}
\medskip

For this, consider  the $d$-dimensional unit sphere $S_d=\{\xxx\in \R^d: \|\xxx\| =1\}$. Then, we select a finite set of pairwise distinct points $(\zz_1, \zz_2, ...\zz_N) \in (S_d)^N $ for some integer $N\geq1$ such that the convex hull of $\{\zz_1, \zz_2, ...\zz_N\}$ contains the $d$-dimensional  ball $B({\bf 0},  1/2)$.

Let us choose $\ep_c>0$ so small that :
\begin{itemize}
\smallskip\smallskip
\item 
$0<\ep_c \leq \frac{1}{1000}\min (\|\zz_i-\zz_j\|, \ i\neq j,  \ i,j\in \{1,...,N\} )$.
\smallskip\smallskip
\item
Setting for every $i\in \{1,...,N\}$
\begin{equation}
\label{defCI}
C_i = S_d \cap B(\zz_i,\ep_c),
\end{equation}
then  for any choice of $\zz'_i \in C_i$, the convex hull of $\{\zz'_1, \zz'_2, ...\zz'_N\}$ contains the ball $B({\bf 0},  1/4)$.

\end{itemize}

%%%%%%%%%%%%%%%%%%%%%%%%%%%%%
 
%%%%%%%%%%%%%%%%%%%%%%%%%%%%%%%%%%
\begin{proposition}
\label{gototheaxis}
If $h_f(\xxx) = h$, then there exists $i\in \{1,...,N\}$ such that for every $\zz'_i \in C_i$ (see \eqref{defCI}), the restriction of $f $ to the straight line passing through $\xxx$ parallel to  the vector $\zz'_i $ has a pointwise H\"older exponent equal to  $h$.
\end{proposition}
%%%%%%%%%%%%%%%%%%%%%%%%%%%%%%%%%%

%%%%%%%%%%%%%%%%%%%%%%%%%%%%%%%%%%
\begin{proof}
Let $n$ be such that $n\leq h<n+1$. 
We assume without loss of generality that $\xxx={\bf 0}$, $f({\bf 0}) =0$, $D^{k}f({\bf 0},...,{\bf 0})=0$ for every $k\in \{1,..., n\}$.   Let $\ep>0$. By definition, for every $\xxx$ close to ${\bf 0}$, $|f(\xxx)| \leq |\xxx|^{h-\ep}$, and there exists   a sequence $(\xxx_n = (x_{n,1},...,x_{n,d}))_{n\geq 1}$ of elements in $\R^d$, converging to ${\bf 0}$,  such that
\begin{equation}
\label{defyn}
|\xxx_n|^{h+\ep} \leq | f(\xxx_n) | \leq |\xxx_n|^{h-\ep}.
\end{equation}

Consider such an element $\xxx_n$, and the sets  $(C_{n,i}:= 4\|\xxx_n\|\cdot C_i)_{i=1,..., N}$. Let us prove that there exists $i_n\in \{1,..., N\}$ such that for every 
$n\in \N$  and
$\zz'_{i_n } \in C_{n,i_n}$, $|f(\zz'_{i_n})|\geq (|\zz'_{i_n}|/4)^{h+\ep}$.

Assume first  that for every $i\in \{1,..., N\}$, there exists  $\zz'_i \in C_{n,i}$ such that $|f(\zz'_i)|< |f(\xxx_n)|$.
By construction, the ball $B({\bf 0}, \|\xxx_n\|)$ is included in the convex hull of these points $(\zz'_i)_{i=1,..., N}$. By convexity, this would imply that $|f(\xxx_n)| \leq \max( |f(\zz'_i)|: i=1,...,N)$, hence a contradiction.

Hence, there exists an  $i_n\in \{1,...,N\}$ such that  for every $\zz'_{i_n} \in C_{n,{i_n}}$, $|f(\zz'_{i_n})|\geq |f(\xxx_n)| \geq |\xxx_n|^{h+\ep} = (|\zz'_{i_n}|/4)^{h+\ep}$.

\smallskip

 Turning to a subsequence, one can assume that $i_n$ is constant, equal to $i\in\{1,...,N\}$.
 
 Now, as a consequence of what precedes, for every vector $\zz'_i\in C_i$, there exists an infinite number of values $(r_n=|\xxx_n|)_{n\geq 1}$ such that $|f(r_n \zz'_i)|\geq (|r_n \zz'_i|/4)^{h+\ep}$. Hence, the restriction of $f$ to the straight line passing through ${\bf 0}$ parallel to  $\zz'_i$ has a pointwise H\"older exponent less than $h+\ep$. Since this holds for every $\ep>0$, and obviously this exponent is bounded below by $h$ (i.e. the exponent of $f$), one concludes that this restriction has exactly exponent $h$. 
 \end{proof}
%%%%%%%%%%%%%%%%%%%%%%%%%%%%%%%%%%

%%%%%%%%%%%%%%%%%%%%%%%%%%%%%%%%%%
%%%%%%%%%%%%%%%%%%%%%%%%%%%%%%%%%%
%%%%%%%%%%%%%%%%%%%%%%%%%%%%%%%%%%
%%%%%%%%%%%%%%%%%%%%%%%%%%%%%%%%%%
%%%%%%%%%%%%%%%%%%%%%%%%%%%%%%%%%%
\section{First typical properties of continuous convex functions}

Based on Lemma \ref{*djdiffd} it is very easy to see that the typical function in $ {\CC}$ is continuously differentiable on $ {(0,1)^d}$. This was proved in \cite{Gruber,Howe} for instance. We give another proof for completeness. 
%%%%%%%%%%%%%%%%%%%%%%%%%%%%%%%%%%
 \begin{proposition}
 \label{*gendiff}
 There is a dense $G_{{\delta}}$ set $ {{\cal G}} $ in $ {\CC}$   such that   every $f\in  {{\cal G}}$   is continuously differentiable  on $ {(0,1)^d}$.
 \end{proposition}
%%%%%%%%%%%%%%%%%%%%%%%%%%%%%%%%%%
 
%%%%%%%%%%%%%%%%%%%%%%%%%%%%%%%%%%
 \begin{proof}
 By convexity, the partial derivatives $ {\partial}_{j,\pm} f( \xxx)$ exist
for any $f\in {\CC}$, $ \xxx\in  {(0,1)^d}$   and $j\in {\{1,...,d \}}.$

Since $\CC$ is separable, one can choose a   sequence of convex  functions $\{f_{m}:m=1,... \}$ dense  in $\CC$. In addition, by Remark \ref{rk2}, one can assume that   all these functions $f_m$ are   $ {{C^ {\infty}}} (\zu^d )  $ functions.

By uniform continuity of all the partial derivatives of $f_m$,  there is a $ {\delta}_{n,m}>0$  such that  for every $j$,  for every $
  \xxx, \xxx'\in  {[0,1]^d} $,
\begin{equation}\label{*R8*a}
| {\partial}_{j}f_{m}( \xxx)- {\partial}_{j}f _{m}( \xxx')|<\frac{1}{n}, \ \ \mbox{when }| \xxx- \xxx'|< {\delta}_{n,m}.
\end{equation}
Applying Lemma \ref{*djdiffd}, it is possible to choose    $0< {\varrho}_{n,m}<\frac{1}{n+m}$
 such that if $f\in B_{\|\cdot\|}(f_{m}, {\varrho}_{n,m})$ then for every $j\in {\{1,...,d \}}$, if  $\xxx=( {{x_ {1} ,...,x_ {d}}}) \in \zu^d$ with   $x_{j}\in [\frac{1}{n},1-\frac{1}{n}]$,  then 
 \begin{equation}\label{*R8*b}
 | {\partial}_{j,\pm} f( \xxx)- {\partial}_{j}f _{m}( \xxx)|< \frac{1}{n}.
\end{equation}

Let us introduce the sets  
$${\cal G}_{n}=\bigcup_{m=1}^{{\infty}}B_{\|\cdot\|}(f_{m} , {\varrho}_{n,m})  \mbox { \ and \ } {\cal G} =\bigcap_{n=1}^{{\infty}}{\cal G}_{n} .$$
 It is clear that ${\cal G} $
is a dense $G_{{\delta}}$ set in $ {\CC}$.

We prove that any $f\in{\cal G} $ is continuously differentiable.  

By definition,    there exists an infinite sequence of integers  $(m_{n})_{n\geq 1}$  such that $f\in B_{\|\cdot\|}(f_{m_{n}}, {\varrho}_{n,m_{n}}).$

Fix $j\in \{1,..,d\}$, and focus on the $j$-th partial derivatives.
Combining inequalities  \eqref{*R8*a} and \eqref{*R8*b},  if $ \xxx, \xxx' \in{[0,1]^d}$, with
$x_{j},x_{j}'\in[\frac{1}{n},1-\frac{1}{n}]$ and $|| \xxx- \xxx'||< {\varrho}_{n,m_{n}}$, then
$$| {\partial}_{j,\pm} f( \xxx)- {\partial}_{j,\pm} f( \xxx')|<| {\partial}_{j}{f_{m_{n}}}( \xxx)- {\partial}_{j}{f_{m_{n}}}( \xxx')|+\frac{2}{n}<\frac{3}{n}.$$
The $\pm$ is the above inequality means that any choice of left or right derivative can be made. 

From this, letting $n\to {\infty}$, it follows easily that $ {\partial}_{j,+} f( \xxx)= {\partial}_{j,-} f( \xxx)$, hence $ {\partial}_{j} f$ is continuous on $ {(0,1)^d}$.
 \end{proof}
 %%%%%%%%%%%%%%%%%%%%%%%%%%%%%%%%%%

 However,  the next lemma shows that  typical functions $f$ in $ {\CC}$ are not differentiable on $ {[0,1]^d}$. The problem comes from the boundary of the domain.
 
%%%%%%%%%%%%%%%%%%%%%%%%%%%%%%%%%%
 \begin{proposition}\label{*ddoo}
 There is a dense $G_{{\delta}}$ set ${\cal G}_{{\infty}}$ in $ {{\cal G}}$
  such that   every $f\in{\cal G}_{{\infty}}$ satisfies the following: for every $j\in  {\{1,...,d \}}$, for every 
  $x_{i}\in[0,1]$ with $i\in {\{1,...,d \} \setminus \{j \}}$, one has
  \begin{equation}
  \label{*R10a*a}
   {\partial}_{j,+}f(x_{1},...,x_{j-1}, 0 ,x_{j+1},...,x_{d}) =- {\infty}
  \end{equation}
 and
  \begin{equation}\label{*R10a*b}
   {\partial}_{j,-}f(x_{1},...,x_{j-1}, 1 ,x_{j+1},...,x_{d}) =+ {\infty}.
  \end{equation}
Moreover,
 \begin{equation}\label{*R10a*c}
 h_{f}(x_{1},...,x_{j-1}, 0,x_{j+1},...,x_{d})=0=
 h_{f}(x_{1},...,x_{j-1}, 1 ,x_{j+1},...,x_{d}).
 \end{equation}
 \end{proposition}
%%%%%%%%%%%%%%%%%%%%%%%%%%%%%%%%%%

%%%%%%%%%%%%%%%%%%%%%%%%%%%%%%%%%%
 \begin{proof}
 We are going to show  that for a fixed $j$, there is a dense $G_{{\delta}}$ set
 ${\cal G}^{0,j}$ in $ {{\CC}}$  such that if $f\in{\cal G}^{0,j}$
 then \eqref{*R10a*a} and the first equality in 
 \eqref{*R10a*c} holds.
  
 \medskip
  
Without loss of generality, one considers  $j=1$. As in the proof of Theorem \ref{*gendiff}, one chooses a sequence of $C^\infty$
 functions  $(f_{m})_{m\geq 1}$. We also select integer constants $M_{m}\geq 1$
such that $| {\partial}_{1}f_{m}|\leq M_{m}$ on $ {[0,1]^d}$.

For every integer $l\geq 1$, let us introduce the  mapping $ {\varphi}_{l}$ defined as
\begin{eqnarray*}
 {\varphi}_{l}( {{x_ {1} ,...,x_ {d}}})= \begin{cases} -l^{l-1}x_{1} &  \mbox { if }0\leq x_{1}\leq l^{-l},\\
 -l^{-1} &  \mbox { if } l^{-l}\leq x_{1}\leq 1.
\end{cases}
\end{eqnarray*}

Clearly, $ {\varphi}_{l}$ is convex and for any fixed $n$ the functions
$f_{m}+ {\varphi}_{n\cdot m\cdot M_{m}}$, $m=1,...$ are dense in $ {\CC}$.
Set $${\cal G}_{n}^{0,1}=\bigcup_{m=1}^{{\infty}}B_{\|\cdot\|}(f_{m}+ {\varphi}_{n\cdot m\cdot M_{m}},
(n\cdot m\cdot M_{m})^{-n\cdot m\cdot M_{m}}),$$ and 
$${\cal G}^{0,1}=\bigcap_{n=1}^{{\infty}} {\cal G}_{n}^{0,1}.$$

We prove that if $f\in{\cal G}^{0,1}$, then \eqref{*R10a*c} and hence \eqref{*R10a*a} hold.  By definition, there exists an infinite sequence of integers  $(m_{n})_{n\geq 1}$  such that 
$$f\in B_{\|\cdot\|} \big(f_{m_{n}}+ {\varphi}_{n\cdot m_n\cdot M_{m_n}},(n\cdot m_n\cdot M_{m_n})^{-n\cdot m_n\cdot M_{m_n}}\big).$$
For ease of notation, set $l_{n}=n\cdot m_n\cdot M_{m_n}$, that is
$$f\in B_{\|\cdot\|}(f_{m_{n}}+ {\varphi}_{l_{n}},l_{n}^{-l_{n}}).$$

Consider $\xxx\in \zu^d$, with $x_1=0$.  Then for every $n\geq 5$, 
\begin{eqnarray*}
\label{*R10c*a}
&&f(0,x_{2},...,x_{d})-f(l_{n}^{-l_{n}},x_{2},...,x_{d}) \\
& \geq  &  (f_{m_{n}} + {\varphi}_{l_{n}})(0,x_{2},...,x_{d}))-(f_{m_{n}} + {\varphi}_{l_{n}})(l_{n}^{-l_{n}},x_{2},...,x_{d}))  -2\cdot l_{n}^{-l_{n}}\\
  & \geq & -M_{m_{n}} l_{n}^{-l_{n}} + l_n^{-1} -2\cdot l_{n}^{-l_{n}}\\
 & = & 
 l_{n}^{-1}\Big (1-\frac{M_{m_{n}}}{l_{n}}l_{n}^{-l_{n}+2}-\frac{2}{l_{n}}l_{n}^{-l_{n}+1}\Big ) \\
& > &\frac{1}{2l_{n}},
\end{eqnarray*}
where we have used the boundedness of $\partial f_{m_{n}}$ by $M_{m_n}$ and the fact that $l_n>\!\!> M_{m_n} $.
Hence, for any $ {\alpha}>0$ we have
$$\frac{f(0,x_{2},...,x_{d})-f(l_{n}^{-l_{n}},x_{2},...,x_{d})}{l_{n}^{-l_{n} {\alpha}}}>\frac{1}{2}l_{n}^{l_{n} {\alpha}-1} ,$$
which tends to infinity when $n$ goes to infinity. Hence \eqref{*R10a*a} holds true, and one also deduces that $h_f(0,x_{2},...,x_{d}) = 0$.

\medskip

Similar arguments yield $G_\delta$ sets ${\cal G}^{0,j}$ for $j>1$ and ${\cal G}^{1,j}$ for $j=1,...,d$ such that the functions $f$ in these sets satisfy the corresponding equalities in (\ref{*R10a*a}-\ref{*R10a*c}).
 
 Finally, the set 
 $$ {{\cal G}}_\infty =  \bigcap_{j= 1}^d  {\cal G}^{0,j}\cap{\cal G}^{1,j}$$
 satisfies the conditions of Proposition \ref{*ddoo}.
  \end{proof}
 %%%%%%%%%%%%%%%%%%%%%%%%%%%%%%%%%%

We conclude this section by proving that typical convex functions have only pointwise exponents less than 2.
%%%%%%%%%%%%%%%%%%%%%%%%%%%%%%%%%%%
\begin{proposition}\label{*lemupperh*}
There exists a $G_\delta$-set $\mathcal{G}^2$ such that for every 
$f\in \mathcal{G}^2$,  $E_f(h) = \emptyset$ for every $h>2$.
\end{proposition}
%%%%%%%%%%%%%%%%%%%%%%%%%%%%%%%%%%%

Before proving Proposition \ref{*lemupperh*}, we introduce some perturbation functions already used in Section 4 of \cite{BSJMAA}.

%%%%%%%%%%%%%%%%%%%%%%%%%%%%%%%%%%%
\begin{definition}
\label{defgamma}
For every 
$l\in  {\ensuremath {\mathbb N}}$, the function  
 $\ggg_{l}: {[0,1]} \to  {[0,1]}$ is defined   as follows:  
\begin{itemize}
\item $\gamma_l$ is continuous,
\item
For every integer $j=0,...,2^{l^{2}}-1$,  if $x_{1}\in [j2^{-l^{2}},(j+1)2^{-l^{2}}-2^{-l^{4}}]$,  one sets $\ggg_{l}(x_{1})=j2^{-l^{2}-l}$. So $\gamma_l$ is constant on these intervals,

  \item
 $\ggg_{l}(1)=2^{-l}$,
 
 \item
For every integer $j=0,...,2^{l^{2}}-1$,  the mapping   $\ggg_{l}$  is affine on the intervals  $[
 (j+1)2^{-l^{2}}-2^{-l^{4}},(j+1)2^{-l^{2}}]$,

\end{itemize}
\end{definition}
 %%%%%%%%%%%%%%%%%%%%%%%%%%%%%%%%%%%
These functions $\ggg_l$ are continuous, ranging from $0$ to $2^{-l}$, and  are strictly increasing only on the $2^{l^2}$ many very small, uniformly distributed, intervals of length $2^{-l^4}$.

%%%%%%%%%%%%%%%%%%%%%%%%%%%%%%%%%%%
\begin{proof}
 We start by selecting a   set of $ {{C^ {\infty}}}$ functions $\{ \displaystyle  f_{m}:m=1,2,...\}$ which is dense in $ {\CC}$.
We choose an integer $M_{m,3}\geq 1$ such that the second and  third partial derivatives $\dd_{1}^{3}f_{m}$ with respect to the first variable $x_1$ of $f_m$ satisfy  $|\dd_{1}^{2}f_{m}|+|\dd_{1}^{3}f_{m}|\leq M_{m,3}$.

Observe that the $x_{1}$-partial derivatives $ {\partial}_{1} f_{m}$ of these functions  are monotone in the first variable.

Then,    one introduces the auxiliary functions, depending only on the first variable: for $l\geq 0$, 
\begin{equation}\label{*13**a}
{\overline \ggg}_{l}(x_{1}, x_{2},...,x_{d})=\ggg_{l}(x_{1}).
\end{equation}

The perturbation functions are defined for $l\geq 0$ by
\begin{equation}\label{*14**a}
{\overline f}_{l}( \xxx)={\overline f}_{l}(x_{1}, x_{2},...,x_{d})=\int_{0}^{x_{1}}{\overline \ggg}_{l}(t,x_{2},...,x_{d})dt=\int_{0}^{x_{1}}\ggg_{l}(t)dt.
\end{equation}

Next we  apply Lemma \ref{*djdiffd}  to the functions $f_{m}+ {\overline f}_{ m+M_{m,3}+n}$ with $ {\varepsilon}= {\varepsilon}_{m,n}:= 2^{-(  m+M_{m,3}+n) ^8}$ and $j=1$.
There exists a constant, denoted by $\varrho_{m,n}>0$,  such that if
$f\in B_{\|\cdot\|}(f_{m}+{\overline f}_{ m+M_{m,3}+n},\varrho_{m,n})$, then
for any $x_{1}\in[{\varepsilon}_{m,n},1- {\varepsilon}_{m,n} ]$ and $x_{j}\in [0,1]$ for $j=2,...,d$, one has
\begin{equation}\label{*2**hiv}
| {\partial}_{1,\pm}f({\mathbf x})- {\partial}_{1}(f_{m}+{\overline f}_{ m+M_{m,3}+n})({\mathbf x})|<  {\varepsilon}_{m,n} .
\end{equation}
Without limiting generality one  can assume that $\varrho_{m,n}\leq {\varepsilon}_{m,n}$.

Set
$$ {{\cal R}}_{n}=\bigcup_{m=1}^{{\infty}} B_{\|\cdot\|}(f_{m}+{\overline f}_{ m+M_{m,3}+n},\varrho_{m,n}) .$$
It is not difficult to see that $ {{\cal R}}_{n}$ is open and dense in $ {\CC}$. 
Suppose that $\cag_{\oo}$ is the dense $G_{\ddd}$
set from Proposition  \ref{*ddoo}. Since  $\cag_{\oo}$ is a subset of $\cag$ from
Proposition \ref{*gendiff}  all $f\in \cag_{\oo} $ are continuously differentiable on
$(0,1)^{d}$. Moreover, the H\"older exponent is zero of these functions
on the boundary of $[0,1]^{d}$. 

 Finally, set 
$$ {{\cal G}}^2= {{\cal G}_\oo}\bigcap  \left(\bigcap_{n=1}^{{\infty}}  {{\cal R}}_{n}\right).$$

By construction, there exists a sequence of integers $(m_{n})_{n\geq 1}$  such that $f\in B_{\|\cdot\|}(f_{m_{n}}+{\overline f}_{ m_n+M_{m_n,3}+n},\varrho_{n,m_{n}})$ for every $n$. 

 For simplification, we set   $l_{n} =m_{n}+M_{m_n,3}+n$,  $\rho_n:= \rho_{m_n,n}$ and $ \ep_n:=  {\varepsilon}_{m_n,n} $, so that for every $n\geq 1$, $\rho_n \leq \ep_n = 2^{-(l_n)^8}$, $f\in B_{\|\cdot\|}(f_{m_{n}}+ { {\overline{f}}}_{l_n},\varrho_{n})$, and for any $x_{1}\in[{\varepsilon}_{n},1- {\varepsilon}_{n} ]$ and $x_{j}\in [0,1]$ for $j=2,...,d$
 \begin{equation}\label{*2**hiv2}
| {\partial}_{1,\pm}f({\mathbf x})- {\partial}_{1}(f_{m_n}+ {{\overline{f}}}_{l_n} )({\mathbf x})|<  {\varepsilon}_{n} .
\end{equation}

Proceeding towards a contradiction, suppose that there exists ${\mathbf x}=(x_{1},x_{2}...,x_{d})\in [0,1]^{d}$ where $h_{f}({\mathbf x})>2$. Since the H\"older exponent of $f$ is zero on the boundary of $[0,1]^{d}$, necessarily  ${\mathbf x}\in (0,1)^{d}$.

Since $h_{f}({\xxx})>2$, one can find $\ep>0$, $D_{2}, C_{{\xxx}}\in \R$  such that for every small $h$, 
\begin{equation}\label{*2**c}
|f({\mathbf x}+h{\mathbf e}_1)-f({\mathbf x})-\partial_1 f({\mathbf x})h-D_{2}h^{2}| \leq  C_{{\mathbf x}} |h|^{2+\ep}.
\end{equation}
Without limiting generality we can suppose that $\ep<1/2$ holds as well.

Consider  the unique integer $j_{n}\in \N$ such that $x_1\in [j_{n}  2^{-l_n^{2}},(j_{n}+1)2^{-l_{n}^{2}})$.   Next we consider two cases depending on whether
$x_1\in [j_{n}  2^{-l_n^{2}},(j_{n}+1)2^{-l_{n}^{2}}-2^{-l_n^{4}}]$, or
$x_1\in [(j_{n}+1)2^{-l_{n}^{2}}-2^{-l_n^{4}},(j_{n}+1)2^{-l_{n}^{2}}]$.

\medskip

{\bf Case 1.} Assume that $x_1 \in [j_n2^{-l_n^{2}},(j_n+1)2^{-l_n^{2}}-2^{-l_n^{4}}]
$ for infinitely many integers $n\geq 1$.

We set  ${\mathbf h}_n=h_{n}{\mathbf e}_1$ with $|{\mathbf h}_n|=|h_{n}|=2^{-l_n^{2}}/4$, such that the first coordinates of  $\xxx$ and $\xxx+{\mathbf h}_n$ both belong to $ [j_n2^{-l_n^{2}},(j_n+1)2^{-l_n^{2}}-2^{-l_n^{4}}]$.

Combining  \eqref{*2**hiv2}, \eqref{*2**c}  and the fact    that  $f\in B_{\|\cdot\|}(f_{m_{n}}+{ {\overline f}}_{l_n},\varrho_{n})$,  we obtain 
\begin{eqnarray}
\nonumber
&&\Big | \Big (f_{m_n}({\mathbf x}+\mathbf{h}_{n}) +{ {\overline f}}_{l_n}({\mathbf x}+\mathbf{h}_{n}) \Big) - \Big( 
f_{m_n}({\mathbf x})+{ {\overline f}}_{l_n} ({\mathbf x})\Big)\\
\nonumber
&& - h_{n}(\partial_1f_{m_n} +\partial_1 { {\overline f}}_{l_n})({\mathbf x}) -D_{2}h_{n}^{2} \
\Big |
 \\
\nonumber
&&\leq  \,C_{{\mathbf x}}|h_{n}|^{2+\ep}+2\rrr_{n }+\ep_{n }|h_{n}|\\
\nonumber
&&\leq  \,C_{{\mathbf x}}|h_{n}|^{2+\ep}+2 \cdot2^{-(l_n)^8} + 2^{-(l_n)^8} |h_{n}|\\
&& \leq \, (C_{{\mathbf x}}+1)|h_{n}|^{2+ \ep},
\label{*2**a}\end{eqnarray}
where the last inequality holds  since $\rho_n\leq \ep_n \leq 2^{-(l_n)^8} <\!\!< |h_n|^3$ for large $n$.

By using the Taylor polynomial estimate of the $C^{\oo}$ function $f_{m_n}$, one deduces that 
\begin{equation}\label{*2**b}
\Big|f_{m_n}({\mathbf x}+\mathbf{h}_{n})-f_{m_n}({\mathbf x})-\partial_1f_{m_n}({\mathbf x})h_{n}-\frac{\partial_1^{2}f_{m_n}({\mathbf x})}{2!}h_{n}^{2} \Big|\leq \frac{M_{m_n,3}}{3!}|h_{n}|^{3}.
\end{equation}

\newcommand\yyy{\mathbf {y}}

Since by its definition $\partial_1 { {\overline f}}_{l_n} (\yyy)=   \ggg_{l_n}(x_{1}) $ (i.e. it is constant)
for any $\yyy$ on the line segment connecting ${\mathbf x}$ and $ {\mathbf x} +\mathbf{h}_{n}$, we also have 
\begin{equation}\label{*2**d}
 \overline{f}_{l_n}({\mathbf x}+\mathbf{h}_{n})-\overline{f}_{l_n}({\mathbf x})-\partial_1\overline{f}_{l_n}({\mathbf x})h_{n} = 0.
\end{equation} 
Using \eqref{*2**a}, \eqref{*2**b} and \eqref{*2**d} we infer
\begin{eqnarray}\label{*3**a}
&&\Big |\frac{\partial_1^{2}f_{m_n}({\mathbf x})}{2!}h_{n}^{2}-D_{2}h_{n}^{2}\Big |\\
\nonumber && \leq (C_{{\mathbf x}}+1)|h_{n}|^{2+\ep}+\frac{M_{m_n,3}}{3!}|h_{n}|^{3}  <
(C_{{\mathbf x}}+2)|h_{n}|^{2+\ep}
\end{eqnarray}
where, using the fact that $\ep<1/2$,
 the last inequality holds if $n$ is sufficiently large since $M_{n,3}\leq l_n \leq 2^{l_n} \leq |h_n|^{-1/2}$ for $n$ large.

 Now take $\overline{{\mathbf h}}_n=8|h_{n}|{\mathbf e}_1.$
 Then \eqref{*2**a} and \eqref{*2**b} used with $\overline{{\mathbf h}}_n$ instead of ${\mathbf h}_n $ 
 for sufficiently large $n$   yield
 \begin{eqnarray}
 \nonumber
 && \Big |\overline{f}_{l_n}({\mathbf x}+\overline{{\mathbf h}}_n)-\overline{f}_{l_n}({\mathbf x})
 -\partial_1\overline{f}_{l_n} ({\mathbf x})|\overline{{\mathbf h}}_n|+ \frac{\partial_1^{2}f_{m_n}({\mathbf x})}{2!}|\overline{{\mathbf h}}_n|^{2}-D_{2}|\overline{{\mathbf h}}_n|^{2}
 \Big| \\
\nonumber
 &&\leq  \frac{M_{m_n,3}}{3!}|\overline{{\mathbf h}}_n|^{3}+(C_{{\mathbf x}}+1 )|\overline{{\mathbf h}}_n|^{2+\ep}\\
\nonumber
 &&
=  |\overline{{\mathbf h}}_n|^{2+\ep}(\frac{M_{m_n,3}}{3!}|\overline{{\mathbf h}}_n|^{1-\ep}+C_{{\mathbf x}}+1 )\\
 \label{*3**b}
\ \  \  &&\leq   |\overline{{\mathbf h}}_n|^{2+\ep}(C_{{\mathbf x}}+2 ).
 \end{eqnarray}
  Now for large $n$,   it follows  from \eqref{*3**a} and $|h_{n}|<|\overline{{\mathbf h}}_n|$ that
\begin{eqnarray}
\label{*4**a}
&& \Big |
 \overline{f}_{l_n} ({\mathbf x}+\overline{{\mathbf h}}_n)-\overline{f}_{l_n}({\mathbf x})-\partial_1\overline{f}_{l_n}({\mathbf x})|\overline{{\mathbf h}}_n|
 \Big| \leq 
 |\overline{{\mathbf h}}_n|^{2+\ep}(2C_{{\mathbf x}}+4 ).
 \end{eqnarray}
 Next we obtain a contradiction by using a lower estimate of the left-hand side
 of \eqref{*4**a}.
 By convexity of $\overline{f}_{l_n} $ we have $\partial_1 \overline{f}_{l_n}  (\yyy)\geq
 \partial_1 \overline{f}_{l_n} ({\mathbf x}) $ when $\yyy$ is on the line segment connecting
 ${\mathbf x}$ and ${\mathbf x}+\overline{{\mathbf h}}_n$. Even more, this interval contains a subinterval of length larger than
 $|\overline{{\mathbf h}}_n|/8=|h_{n}|$ where
 $$\partial_1 \overline{f}_{l_n} (\yyy)={\overline \ggg}_{l_n}(\yyy) = \partial_1\overline{f}_{l_n} ({\mathbf x})+ 2^{-l_n^{2}-l_n}=\ggg_{l_n}(x_{1})+2^{-l_n^{2}-l_n}.$$
 Thus, 
 $$\overline{f}_{l_n} ({\mathbf x}+\overline{{\mathbf h}}_n)-\overline{f}_{l_n} ({\mathbf x})\geq \partial_1\overline{f}_{l_n} ({\mathbf x})|\overline{{\mathbf h}}_n|+\frac{|\overline{{\mathbf h}}_n|}{8}\cdot 2^{-l_n^{2}-l_n}.$$ 
 
 By \eqref{*4**a}, one should have 
 $$\frac{|\overline{{\mathbf h}}_n|}{8}\cdot 2^{-l_n^{2}-l_n} \leq |\overline{{\mathbf h}}_n|^{2+\ep}(2C_{{\mathbf x}}+4 ).$$
 
 Since $8|h_{n}|=|\overline{{\mathbf h}}_n|=2\cdot 2^{-l_n^{2}}$ we would obtain that
 $$2^{-l_n^{2}-l_n}\leq (2\cdot 2^{-l_n^{2}})^{1+\ep}(2C_{{\mathbf x}}+4 ),$$
 a contradiction when $n$ is large.

\medskip

{\bf Case 2.} Suppose that  $x_1 \in [(j_{n}+1)2^{-l_{n}^{2}}-2^{-l_n^{4}},(j_{n}+1)2^{-l_{n}^{2}}]$ for  infinitely many   $n\geq 1$.

We set  ${\mathbf h}_n=h_{n}{\mathbf e}_1$ with $|{\mathbf h}_n|=|h_{n}|=2^{-l_n^{4}}/2$, such that the first coordinates of  $\xxx$ and $\xxx+{\mathbf h}_n$ both belong to $[(j_{n}+1)2^{-l_{n}^{2}}-2^{-l_n^{4}},(j_{n}+1)2^{-l_{n}^{2}}] $.

Since $\rho_n $ and $\ep_n$ are still much smaller than $|h_n|$, equations \eqref{*2**a} and \eqref{*2**b} still hold, but now \eqref{*2**d} is replaced by
\begin{equation}\label{*2**d2}
 \overline{f}_{l_n}({\mathbf x}+\mathbf{h}_{n})-\overline{f}_{l_n}({\mathbf x})-\partial_1\overline{f}_{l_n}({\mathbf x})h_{n} = \frac{\partial_{1}^2\overline{f}_{l_n}({\mathbf x})}{2!} h_{n} ^2 = 2^{l_n^4-l_n-l_n^2}\frac{h_{n} ^2}{2} .
\end{equation}

Using the same arguments as before but with \eqref{*2**d2}, we deduce that
$$
\Big |\frac{\partial_1^{2}f_{m_n}({\mathbf x})}{2!}h_{n}^{2}-D_{2}h_{n}^{2} +2^{l_n^4-l_n-l_n^2}\frac{h_{n} ^2}{2}\Big| <
(C_{{\mathbf x}}+2)|h_{n}|^{2+\ep}.
$$
 This last inequality becomes impossible when $n$ becomes large, since $|\frac{\partial_1^{2}f_{m_n}(\xxx)}{2}|\leq M_{m_n ,3}\leq l_{n}  \leq 2^{l_n^3}$. Hence a contradiction.

\end{proof}
%% 

%%%%%%%%%%%%%%%%%%%%%%%%%%%%%%%%%%
%%%%%%%%%%%%%%%%%%%%%%%%%%%%%%%%%%
%%%%%%%%%%%%%%%%%%%%%%%%%%%%%%%%%%
%%%%%%%%%%%%%%%%%%%%%%%%%%%%%%%%%%
%%%%%%%%%%%%%%%%%%%%%%%%%%%%%%%%%%
%%%%%%%%%%%%%%%%%%%%%%%%%%%%%%%%%%
%%%%%%%%%%%%%%%%%%%%%%%%%%%%%%%%%%
%%%%%%%%%%%%%%%%%%%%%%%%%%%%%%%%%%
%%%%%%%%%%%%%%%%%%%%%%%%%%%%%%%%%%
%%%%%%%%%%%%%%%%%%%%%%%%%%%%%%%%%%
\section{The upper estimate}

We start with the upper bound for the Hausdorff dimensions of the sets $E^{\leq }_f(h)$.

%%%%%%%%%%%%%%%%%%%%%%%%%%%%%%%%%%
\begin{proposition}\label{*C1dim}
If $1< h\leq 2$ and $f\in  {\CC}$, then $\dim  {E_ {f} ^ {\leq }}(h) \leq d+h-2.$
\end{proposition}
%%%%%%%%%%%%%%%%%%%%%%%%%%%%%%%%%%

%%%%%%%%%%%%%%%%%%%%%%%%%%%%%%%%%%
\begin{proof}
It is sufficient to treat the case  $h\in (1,2)$. 

Assume that    $\dim  E_ {f} ^ {\leq }(h) > d+h-2$.

For every $\xxx\in E_ {f} ^ {\leq }(h)$, by Proposition \ref{gototheaxis}, there exists a cone of direction $C_{i_x}$, where   $i_x\in \{1,...,d\}$ such that for every $\zz\in C_{i_x}$, the restriction of $f$ to the straight line passing through $\xxx$  parallel to $\zz$ has exponent less than $h$. 

Let us call $E_i$ the set of elements of $E_ {f} ^ {\leq }(h)$ satisfying this property with $i_x = i\in\{1,...,N\}$. Obviously, $ E_ {f} ^ {\leq }(h)  = \bigcup_{i=1}^N E_i$, so there exists at least  one $i\in \{1,...,N\}$ such that $\dim E_i > d+h-2$.

Let us recall  the following special case of Marstrand's slicing theorem,  
Theorem 10.10 in Chapter 10 of
\cite{MATTILA}.     Recall that $S_d$ is the unit sphere in $\R^d$.

\begin{theorem}
\label{th_marstrand}
Let $E\subset \zu^d$ be a Borel set with Hausdorff dimension $\alpha\in (d-1,d)$. Then for 
 almost every $\zz\in S_d$ (in the sense of $(d-1)$-dimensional ``surface" measure),  there exists a set $E_\zz$ of positive $(d-1)$-dimensional Hausdorff  measure
in the hyperplane orthogonal to $\zz$ such that  for every $\xxx\in E_\zz$,  $\dim E\cap (\xxx+ \R\zz) =\alpha-(d-1)$.
\end{theorem}

Each $C_i$ has non-empty interior in the subspace topology of $S_{d}$,
hence it
is of positive $d-1$-dimensional measure.
Applying Theorem  \ref{th_marstrand} to $E_i$, one can find $\zz\in C_i$   and 
${\overline\xxx}  \in \zu^d$ such that if 
$\cal D = (\overline{\xxx} + \R\zz)$, then $\dim E_i\cap \cal D \geq  d+h-2-(d-1)= h-1$.   

Let us call $g$ the restriction of $f$ to $\cal D$. Then $g$ is still a convex function
of one variable.

By definition of $E_i$, every $\xxx \in \cal D \cap E_i$ satisfies $h_g(\xxx ) \leq h$.

Next, applying   Lemma \ref{*hregd2} to $g$, we deduce that   $\min(h_{g'_+}(\xxx), h_{g'_-}(\xxx)) \leq h -1$, for every $\xxx \in \cal D \cap E_i$.

Hence, at least one of the two sets $E^{\leq}_{g'_+}(h-1)$ and $E^{\leq}_{g'_-}(h-1)$ has Hausdorff dimension strictly greater than $h-1$.

But this is impossible, since both functions $g'_+$ and $g'_-$ are monotone, and for such functions, by \eqref{inegmonotoneb}, the Hausdorff dimension of $E^{\leq}_{g'_+}(h-1) $ and $E^{\leq}_{g'_-}(h-1)$ is necessarily less than $h-1 \in [0,1]$. Hence a contradiction, and the conclusion that $\dim  E_ {f} ^ {\leq }(h) \leq d+h-2$.
\end{proof}
%%%%%%%%%%%%%%%%%%%%%%%%%%%%%%%%%%

%%%%%%%%%%%%%%%%%%%%%%%%%%%%%%%%%%
\begin{proposition}\label{*C2dim}
If $0 \leq  h \leq 1$, $f\in  {\CC}$, then $\dim  {E_ {f} ^ {\leq }}(h) \leq d-1.$
\end{proposition}
%%%%%%%%%%%%%%%%%%%%%%%%%%%%%%%%%%

%%%%%%%%%%%%%%%%%%%%%%%%%%%%%%%%%%
\begin{proof}
The proof is immediate: if $f\in  {\CC}$, the  pointwise exponent of $f$ at any $\xxx \in (0,1)^d$  is necessarily larger or equal than 1.
The remaining points are located on the boundary, whose dimension is $d-1$. And for every $h\in [0,1]$, it is easy to build examples of  convex functions such that $h_f(\xxx)=h$  for every $\xxx$ satisfying $x_1=0$, so the upper bound  $d-1$ for the Hausdorff dimension of  $E_ {f} ^ {\leq }(h) $ is optimal.
\end{proof}
\section{The lower  estimate}
 
 Using Lemma \ref{*djdiffo} it is rather easy to ``integrate" the result about functions in ${\cal M}^{1}= {{\cal M}}$ to obtain the one-dimensional result.
 
 %%%%%%%%%%%%%%%%%%%%%%%%%%%%%%%%%%
\begin{theorem}\label{*genlowero}
There is a dense $G_{{\delta}}$ set $ {{\cal G}}$ in $ {{\mathcal{CC}^1}}$  such that for any
$f\in {{\cal G}}$, and for any 
$1\leq h\leq 2$, $\dim  {E_ {f} ^ { }}(h)= h-1.$
 \end{theorem}
%%%%%%%%%%%%%%%%%%%%%%%%%%%%%%%%%%
% 
% We recall the result about generic monotone functions with a notation
%convenient to prove Theorem \ref{*genlowero}.
% 
% Given $0< {\delta}_{0}<1/2$ we denote by ${\cal M}^{1, {\delta}_{0}}$ the space of monotone
% continuous functions on $[ {\delta}_{0},1- {\delta}_{0}]$.
% Clearly $g\in {\cal M}^{1, {\delta}_{0}}$ is equivalent to
% $g((x- {\delta}_{0})/(1-2 {\delta}_{0}))\in {\cal M}^{1}$.
% 
% \begin{theorem}\label{*genmono}
%Given $0< {\delta}_{0}<1/2$
% there is a dense $G_{{\delta}}$ set ${\cal G}^{{\cal M}, {\delta}_{0}}=\cap_{n=1}^{{\infty}}{\cal G}^{{\cal M}, {\delta}_{0}}_{n}$ in ${\cal M}^{1, {\delta}_{0}}$  such that for any
%$g\in{\cal G}^{{\cal M}, {\delta}_{0}}$ for any
% $0\leq h\leq 1$, $\dim E_{g}^{h}= h$ and $E_{g}^{h}= {\emptyset}$ 
% for $h>1$.
% \end{theorem}
% 
%%%%%%%%%%%%%%%%%%%%%%%%%%%%%%%%%
 \begin{proof} 
 Suppose $0< {\delta}_{0} <1/2$ is fixed. 
 
 Recalling the result  on typical monotone continuous  functions, there exists a $G_\delta$ set of functions  ${\cal G}^{{\cal M}, {\delta}_{0}} $ in  the set $\mathcal{M}^{1,\delta_0} := \{f:[\delta_0,1-\delta_0] \to \R, \ f \mbox{ monotone}\}$ such that every $f\in {\cal G}^{{\cal M}, {\delta}_{0}} $ satisfies \eqref{inegmonotone}.
 
 Let us write ${\cal G}^{{\cal M}, {\delta}_{0}} = \bigcap _{n\geq 1} {\cal G}^{{\cal M}, {\delta}_{0}}_n$, where ${\cal G}^{{\cal M}, {\delta}_{0}}_n$ is a dense open set in $\mathcal{M}^{1,\delta_0}$.
 
 Let us choose a dense sequence $(g_{n,m})_{m=1}^{{\infty}}$ in ${\cal G}^{{\cal M}, {\delta}_{0}}_{n}$.  
 
 By taking antiderivatives of the elements of this sequence and by a suitable definition on the intervals $[0,{\delta}_{0})\cup (1-{\delta}_{0},1]$, one can choose a   sequence of convex functions $(f_{n,m,k})_{m,k=1}^{+\infty} $ which is dense in $ \cal {CC}^1 $  such that for every $m,k\geq 1$, $f_{n,m,k}'(x)=g_{n,m}(x)$ for $x\in[ {\delta}_{0},1- {\delta}_{0}]$. These functions $f_{n,m,k}$ are continuously differentiable on $[\delta_0,1-\delta_0]$.

Now, choose $ {\varepsilon}_{n,m}>0$ such that
$$
 B_{\|\cdot\|}(g_{n,m}, {\varepsilon}_{n,m}) {\subset} {\cal G}^{{\cal M}, {\delta}_{0}}_{n},
$$
 where the ball $B_{\|\cdot\|}$  is taken in the set  ${{{\cal M}}^{1, {\delta}_{0}}}$ using the $L^\infty$-norm.

   Lemma \ref{*djdiffo} gives the existence  of $ {\varrho}_{n,m,k}>0$  such that  for every 
 $f\in B_{\|\cdot\|}(f_{n,m,k}, {\varrho}_{n,m,k}) \subset \mathcal{CC}^1$, the inequality
$$
 |f'_{\pm}(x)-g_{n,m}(x)|< {\varepsilon}_{n,m}
$$
 holds for all $x\in[ {\delta}_{0},1- {\delta}_{0}]$.

Let us now introduce 
$$ {{\cal G}}_{n}=\bigcup_{n,k}B_{\|\cdot\|}(f_{n,m,k}, {\varrho}_{n,m,k}).$$

By construction, $ {{\cal G}}_{n}$ is dense in $ {{\cal {CC}}^1}.$

 Proposition  \ref{*gendiff}  yields the existence of a  dense $G_{{\delta}}$ set   $ {{\cal D}}$ in $ \cal {CC}^1$   consisting of 
 functions $f$  continuously differentiable on $(0,1)$.

We finally set
$$ {{\cal G}}= {{\cal D}}\bigcap \left (\bigcap_{n=1}^{{\infty}} {{\cal G}}_{n} \right).$$
 Clearly, $ {{\cal G}}$ is a dense
 $G_{{\delta}}$ set in $ {{\cal {CC}^1}}.$
 
 Suppose that $f\in {{\cal G}}$, and set $g=f'$ on $(0,1)$ and $g_{0}=g|_{[ {\delta}_{0},1- {\delta}_{0}]}$.
 
 Then $g_{0}\in \bigcap_{n=1}^{{\infty}}{\cal G}^{{\cal M}, {\delta}_{0}}_{n}$ 
 and  equation \eqref{inegmonotone} implies that for any $1\leq h\leq 2$, $\dim E_{g_{0}}(h-1)=h-1$.
 
 But  Lemma  \ref{*hregd} yields, $E_{g_{0}}(h-1)= E_{f|_{[ {\delta}_{0},1- {\delta}_{0}]}}({h})$, hence $\dim E_{f}(h)=h-1$ for any $1\leq h\leq 2$.
 
 %\medskip
 
 Since $ {\delta}_{0}$ can be chosen arbitrarily small by taking a sequence of $ {\delta}_{0}$'s tending to zero, we can conclude the proof of the theorem.
 \end{proof}
 %%%%%%%%%%%%%%%%%%%%%%%%%%%%%%%%%%

Next we turn to the higher dimensional case.

%%%%%%%%%%%%%%%%%%%%%%%%%%%%%%%%%%
\begin{theorem}\label{*genlowerd}
There is a dense $G_{{\delta}}$ set $ {{\cal G}}$ in $ {\CC}$  such that 
for any $f\in {{\cal G}}$ and   $1\leq h \leq2$ , one has $\dim E_{f}(h)\geq h+d-2$. In addition,  $E_{f}(h)= {\emptyset}$ for $h>2$, $E_{f}(h)\cap  {[0,1]^d}= {\emptyset}$
for $0<h<1$, and  $ {\partial} ({[0,1]^d})=E_{f}(0)$.
\end{theorem}
%%%%%%%%%%%%%%%%%%%%%%%%%%%%%%%%%%

%%%%%%%%%%%%%%%%%%%%%%%%%%%%%%%%%%
\begin{proof}
Now instead of the functions, we can ``integrate" the proof used in Section 4 of \cite{BSJMAA}.
The idea is again to reduce the problem to the one-dimensional case.
We select one coordinate direction, for ease of notation  the first, the $x_{1}$-axis.
We use in our proof the perturbation functions which are constant in the directions of the coordinate axes $x_{j}$, $j=2,...,d$ already used  in this paper, given in Definition \ref{defgamma}.

Let us select a dense set of $ {{C^ {\infty}}}$ functions $\{ \displaystyle  f_{m}:m=1,2,...\}$ which is dense in $ {\CC}$.
The $x_{1}$-partial derivatives, $ {\partial}_{1} f_{m}$, of these functions will be denoted by $g_{m}$.  The important feature of these functions 
$g_{m}$ is the fact that they are monotone increasing in the $x_{1}$-variable. As our example at the beginning of the paper shows (see Remark \ref{rk1}),  these functions are not necessarily monotone in the other variables, and this is why one cannot ``integrate" simply the MISV genericity results.

Now, as in \cite{BSJMAA} and   in the proof of Proposition \ref{*lemupperh*}, we use the functions  ${\overline \ggg}_{l}$ and the perturbations ${\overline f}_{l}$ defined in \eqref{*13**a} and \eqref{*14**a}.

Next,  apply Lemma \ref{*djdiffd}  to the functions $f_{m}+ {\overline f}_{m+n}$ with $ {\varepsilon}= \ep_{n+m}$ and $j=1$.
There exists a constant, denoted by $\varrho_{m,n}>0$,  such that if
$f\in B_{\|\cdot\|}(f_{m}+{\overline f}_{m+n},\varrho_{m,n})$, then
for any $x_{1}\in[ {\ep}_{n+m},1- {\ep}_{n+m}]$ and $x_{j}\in [0,1]$ for $j=2,...,d$, one has
$$| {\partial}_{1,\pm}f({\mathbf x})- {\partial}_{1}(f_{m}+{\overline f}_{m+n})({\mathbf x})|< {\ep}_{n+m},$$
that is,
\begin{equation}\label{*R19*a}
| {\partial}_{1,\pm}   f({\mathbf x})-(g_{m}+ {\overline \ggg}_{m+n}({\mathbf x}))|< {\ep}_{n+m}
\end{equation}
holds.
Without limiting generality one  can assume that $\varrho_{m,n}\to 0$
if $n$ is fixed and $m\to {\infty}$.

Set
$$ {{\cal R}}_{n}=\bigcup_{m=1}^{{\infty}} B_{\|\cdot\|}(f_{m}+{\overline f}_{m+n},\varrho_{m,n}) .$$
It is not difficult to see that $ {{\cal R}}_{n}$ is open and dense in $ {\CC}$. Denote by $ {{\cal D}}$ a dense $G_{{\delta}}$ set in $ {\CC}$, which consists of functions differentiable on $ {(0,1)^d}$, and such that (according to Proposition  \ref{*lemupperh*}) these functions also have nowhere a pointwise exponent strictly greater than 2.

 Finally, set 
$$ {{\cal G}}= {{\cal D}}\bigcap  \left(\bigcap_{n=1}^{{\infty}}  {{\cal R}}_{n}\right).$$

Let $f\in  {{\cal G}}$, and  $1\leq h \leq 2$. We are going to prove that $\dim E_{f}(h)\geq h+d-2$, by reducing the argument to a situation already totally taken care of in \cite{BSJMAA}.

By construction, there exists a sequence of integers $(m_{n})_{n\geq 1}$  such that $f\in B_{\|\cdot\|}(f_{m_{n}}+{\overline f}_{m_{n}+n},\varrho_{n,m_{n}})$ for every $n$.

Set $g= {\partial}_{1} f.$

By \eqref{*R19*a}, when $x_{1} \in [ \ep_{n+m_{n}},1- \ep_{n+m_{n}}]$ and $x_{j}\in [0,1]$
for $j=2,...,d$, one has for $\xxx=(x_1,x_2,...,x_d)$ 
\begin{equation}\label{*R20*a}
|g({\mathbf x})-(g_{m_{n}}+{\overline \ggg}_{m_{n}+n})({\mathbf x})|< \ep_{n+m_{n}}.
\end{equation}

Given $0< {\delta}_{0}<1/10$, choose an integer $n_{1}$  such that $ {\ep}_{n_{1}}< {\delta}_{0}/100.$

Select an increasing subsequence $(n_{k})_{k\geq 1}$ such that the following conditions are fulfilled:  put $l_{k}=m_{n_{k}}+n_{k}$.  One assumes that $$l_k > 2^ k , ((l_k)^2 + l_k)k + 1 < (l_k)^4, \  2 ^{-((l_{k-1})^2+l_{k-1})(k-1)-1} > 100 \cdot  2^{-(l_k)^2}$$
and if $D_k = 2^{(l_k)^2} \cdot 2^{-((l_k)^2+l_k)k-2} < 1 $, then one also assumes that $k$ is so large that
$$ D_1 \cdots  D_{k-1} > 2^{-l_k}.$$
These conditions are equations (27) and (28) in \cite{BSJMAA}.

Denote by $ {\varphi}_{k}$ the restriction of $g_{m_{n_{k}}}$
onto $[ {\delta}_{0},1- {\delta}_{0}]\times [0,1]^{d-1}$ and by ${\widetilde {g}}_{l_{k}}$ the restriction of ${\overline \ggg} _{m_{n_{k}}+n_{k}}$ onto $[ {\delta}_{0},1- {\delta}_{0}]\times [0,1]^{d-1}$.
For ease of notation for the restriction of $g$ onto $[ {\delta}_{0},1- {\delta}_{0}]\times [0,1]^{d-1}$
we  will still use the notation $g$.

Then $g\in B_{\|\cdot\|}( {\varphi}_{k}+{\widetilde {g}}_{l_{k}}, {\ep}_{l_{k}})$ for all $k=1,...$, where  the ball is taken  with respect to the supremum norm in the space of continuous functions   monotone in the first variable.

\medskip

Now, we are exactly in the context of our previous article \cite{BSJMAA}, in which we proved the following sequence of propositions (cf  Propositions 13-18 of \cite{BSJMAA}). We reproduce the definitions given in \cite{BSJMAA} and the associated propositions.

For every $h\in (1,2)$ and $k\geq 2$, let
$$ F_{h-1,k}=\bigcup_{j=0}^{2^{(l_k)^2}-1} \left[  (j+1)2^{-(l_k)^2}  - 2^{\frac{(l_k)^2+l_k}{h-1}},  (j+1)2^{-(l_k)^2}  - \frac{1}{2} 2^{\frac{(l_k)^2+l_k}{h-1}}\right] .$$
For $h=1$, set $ F_{h-1,k}=F_{0,k}$ as
$$  F_{0,k}=  \bigcup_{j=0}^{2^{(l_k)^2}-1} \left[  (j+1)2^{-(l_k)^2}  - 2^{k{((l_k)^2+l_k)}{}},  (j+1)2^{-(l_k)^2}  - \frac{1}{2} 2^{{k((l_k)^2+l_k)}{}}\right]. $$

One has \cite{BSJMAA}:
%%%%%%%%%%%%%%%%%%%%%%%%%%%%%%%%%%%
\begin{proposition}
\label{propmin}
For $h\in [1,2)$, let $k_h = \max (3, [1/h]+2)$  and 
$$F_{h-1} = \bigcap_{k\geq  k_h } F_{h-1,k} \ \subset \  [{\delta}_{0},1- {\delta}_{0}] .$$
For every $\xxx\in F_{h-1}\times \zu^{d-1}$, $h_g(\xxx)\leq h-1.$

In particular, $\dim E_g(0)=d-1$.
\end{proposition}
%%%%%%%%%%%%%%%%%%%%%%%%%%%%%%%%%%%

%%%%%%%%%%%%%%%%%%%%%%%%%%%%%%%%%%%
\begin{proposition}
For $h\in [1,2)$,  there exists a probability measure $\mu_{h-1}$ such that $\mu_{h-1}( F_{h-1}\times \zu^{d-1}) =1$  and for every $\xxx\in  F_{h-1}\times \zu^{d-1}$, 
\begin{equation}\label{*R22*a}
\liminf_{r\to 0+}\frac{\log  {{ {\mu}}_{h-1}}\big (B(\xxx,r) \big)}
{\log r}\geq d+h-2.
\end{equation}  

\end{proposition}
%%%%%%%%%%%%%%%%%%%%%%%%%%%%%%%%%%%
Hence, using the mass distribution principle, one deduces from \eqref{*R22*a} that $\dim (F_{h-1}\times \zu^{d-1} ) \geq d+h-2$. From the last two propositions, one deduces that for every $h\in [1,2]$, 
$$\dim E_g^{\leq } (h-1) \geq  d+h-2.$$ 

%
%\end{proposition}
%%%%%%%%%%%%%%%%%%%%%%%%%%%%%%%%%%%%
%
%%%%%%%%%%%%%%%%%%%%%%%%%%%%%%%%%%%%
%\begin{proof}
%We reproduce the proof for convenience.
%Write
%$$ \tilde F_{h-1} = \left( F_{h-1}\times \zu^{d-1} \right)\setminus  \bigcup_{n\geq 1} E_g^{\leq }(h-1-1/n).$$ 
%Since $g$ is monotone in its first variable, the set $\{\xxx\in\zu^d: h_g(x) \leq h\}$ is included in $\dim  E_g^{\leq }(h-1-1/n) \leq h-1-1/n$, hence $\mu_{h-1} \big ( E_g^{\leq }(h-1-1/n) \times \zu^{d-1}\big ) =0$.
%
%Since $\mu_{h-1}( F_{h-1}\times \zu^{d-1}) =1$, one deduces that $\mu_{h-1}(  \tilde F_{h-1} \times \zu^{d-1}) =1$, and  $\dim \tilde F_{h-1} \times \zu^{d-1} =d+h-2$.
%But by construction, the elements in $ \tilde F_{h-1} \times \zu^{d-1}$ satisfy $h_g(\xxx)=h-1$, hence the result.
%\end{proof}
%%%%%%%%%%%%%%%%%%%%%%%%%%%%%%%%%%%%

Finally, Proposition 18 of \cite{BSJMAA} gives:

%%%%%%%%%%%%%%%%%%%%%%%%%%%%%%%%%%%
\begin{proposition}
For every  $h>2$, $ E_g(h-1) = \emptyset$.
\end{proposition}
%%%%%%%%%%%%%%%%%%%%%%%%%%%%%%%%%%%

Finally, we use that the functions $f$   have nowhere a pointwise H\"older exponent greater than 2 (by Proposition \ref{*lemupperh*}), and we finish the proof of Theorem \ref{mainth2}.

 If $\xxx \in    F_{h-1}\times \zu^{d-1} $,  then  $h_g(\xxx) \leq h -1 \in [0,1]$. Necessarily, $ h_g(\xxx) +1\leq h_f(\xxx) \leq 2$. By   Lemma \ref{*hregd}, the only possibility is $h_f(\xxx)  = h_g(\xxx)+1 \leq  (h-1)+1 = h$. Hence $ \dim E^{\leq }_f(h) = \dim E^{\leq}_g(h-1) \geq d+h-2$. Even more, the same argument as above (Proposition \ref{propmin} combined with Proposition \ref{*C1dim}) gives $\mu_{h-1} (E^{\leq }_f(h))>0$. 
 
 Using the upper bound of Theorem \ref{mainth1}, one knows that for every large integer $n\geq 1$, $\dim  (E^{\leq }_f(h-1/n)) \leq h-1/n+d-2$, hence $\mu_{h-1} \big (  E^{\leq }_f(h-1/n) \big ) =0$.  Since 
 $$ E_f(h) =E_f^{\leq} (h) \setminus \bigcap_{n\geq 1}  E^{\leq }_f(h-1/n)  ,$$
 one deduces that $\mu_{h-1} (E_f(h)) >0$, which implies that $\dim E_f(h) \geq d+h-2$. Since the converse inequality also holds true, one concludes that $\dim E_f(h) = d+h-2$.
\end{proof}
%%%%%%%%%%%%%%%%%%%%%%%%%%%%%%%%%%%

%%%%%%%%%%%%%%%%%%%%%%%%%%%%%%%%%%
%%%%%%%%%%%%%%%%%%%%%%%%%%%%%%%%%%
%%%%%%%%%%%%%%%%%%%%%%%%%%%%%%%%%%
%%%%%%%%%%%%%%%%%%%%%%%%%%%%%%%%%%
%%%%%%%%%%%%%%%%%%%%%%%%%%%%%%%%%%
%%%%%%%%%%%%%%%%%%%%%%%%%%%%%%%%%%
%%%%%%%%%%%%%%%%%%%%%%%%%%%%%%%%%%
%%%%%%%%%%%%%%%%%%%%%%%%%%%%%%%%%%
%%%%%%%%%%%%%%%%%%%%%%%%%%%%%%%%%%
%%%%%%%%%%%%%%%%%%%%%%%%%%%%%%%%%%
%%%%%%%%%%%%%%%%%%%%%%%%%%%%%%%%%%
%%%%%%%%%%%%%%%%%%%%%%%%%%%%%%%%%%
%%%%%%%%%%%%%%%%%%%%%%%%%%%%%%%%%%
%%%%%%%%%%%%%%%%%%%%%%%%%%%%%%%%%%
%%%%%%%%%%%%%%%%%%%%%%%%%%%%%%%%%%
%%%%%%%%%%%%%%%%%%%%%%%%%%%%%%%%%%
%%%%%%%%%%%%%%%%%%%%%%%%%%%%%%%%%%
%%%%%%%%%%%%%%%%%%%%%%%%%%%%%%%%%%
%%%%%%%%%%%%%%%%%%%%%%%%%%%%%%%%%%
%%%%%%%%%%%%%%%%%%%%%%%%%%%%%%%%%%
%%%%%%%%%%%%%%%%%%%%%%%%%%%%%%%%%%
%%%%%%%%%%%%%%%%%%%%%%%%%%%%%%%%%%
%%%%%%%%%%%%%%%%%%%%%%%%%%%%%%%%%%

\bibliographystyle{plain}
\bibliography{Bib_tcv}

\begin{thebibliography}{1}

\bibitem{Bay}
F.~Bayart.
\newblock Multifractal spectra of typical and prevalent measures.
\newblock {\em Nonlinearity}, 26:353--367, 2013.

\bibitem{BUC}
Z.~Buczolich and J.~Nagy.
\newblock H\"older spectrum of typical monotone continuous functions.
\newblock {\em Real Anal. Exchange}, 26(1):133--156, 2000.

\bibitem{BuS2}
Z.~Buczolich and S.~Seuret.
\newblock Typical {B}orel measures on $[0,1]^d$ satisfy a multifractal
  formalism.
\newblock {\em Nonlinearity.}, 23(11):7--13, 2010.

\bibitem{BSJMAA}
Z.~Buczolich and S.~Seuret.
\newblock Multifractal spectrum and generic properties of functions monotone in
  several variables.
\newblock {\em J. Math. Anal. Appl.}, 382 (10):110--126, 2011.

\bibitem{Gruber}
P.~M. Gruber.
\newblock Die meisten konvexen {K}\"orper sind glatt, aber nicht zu glatt.
\newblock {\em Math. Ann.}, 229 (3):259--266, 1977.

\bibitem{Howe}
R.~Howe.
\newblock Most convex functions are smooth.
\newblock {\em J. Math. Econom.}, 9 (1-2):37--39, 1982.

\bibitem{MATTILA}
P.~Mattila.
\newblock {\em Geometry of Sets and Measures in Euclidean Spaces}.
\newblock Cambridge University Press, 1995.

\end{thebibliography}

\end{document}